\documentclass[12pt]{article}

\usepackage{authblk}
\usepackage{amssymb}
\usepackage{amsmath}
\usepackage{bbm}
\usepackage{amscd}
\usepackage{amsfonts}
\usepackage{amsthm}
\usepackage{mathrsfs}
\usepackage{verbatim}
\usepackage[colorlinks]{hyperref}
\usepackage{fullpage}
\usepackage{mathdots}
\usepackage{graphicx,subfigure}
\usepackage[english]{babel}
\usepackage{tikz}
\usetikzlibrary{arrows}

\usetikzlibrary{backgrounds,fit,decorations.pathreplacing}
\usetikzlibrary{positioning}
\usetikzlibrary{calc,through,chains}
\usetikzlibrary{arrows,shapes,snakes,automata, petri}

\theoremstyle{definition}
\newtheorem{Theorem}[equation]{Theorem}
\newtheorem{Corollary}[equation]{Corollary}
\newtheorem{Lemma}[equation]{Lemma}
\newtheorem{Proposition}[equation]{Proposition}

\newtheorem{Remark}{Remark}

\newtheorem{Definition}[equation]{Definition}
\newtheorem{Example}[equation]{Example}

\newtheorem{Notation}[equation]{Notation}

\newtheorem{Assumption}[equation]{Assumption}

\numberwithin{equation}{section}
\numberwithin{figure}{section}

\newcommand{\Z}{{\mathbb Z}}
\newcommand{\Q}{{\mathbb Q}}

\newcommand{\N}{{\mathbb N}}

\newcommand{\MM}{\textbf{M}}
\newcommand{\KK}{\textbf{K}}
\newcommand{\OO}{\textbf{\"O}}
\newcommand{\HO}{\textbf{H\"O}}
\newcommand{\CC}{\textbf{\c{C}}}

\begin{document}
% March 10, 2013

\title{Remarks on Enveloping Semigroups}
\date{November 16, 2019}

\author[1]{Mahir Bilen Can}
\affil[1]{mahirbilencan@gmail.com}

\maketitle

\begin{abstract}
The local structures of enveloping semigroups of simple groups are investigated. 
All J-coirreducible connected stabilizer submonoids are determined.
The notion of a navel of a reductive monoid is introduced.
The cross-section lattice of the enveloping monoid is shown to be atomic.  
In type A, the generating series for the number of $G\times G$-orbits is found. 
\vspace{.5cm}

\noindent
\textbf{Keywords:} Enveloping semigroups, asymptotic semigroups, J-coirreducible monoids, navel \\ 
\noindent 
\textbf{MSC:}{ 20M32, 14M27, 06A06} 
\end{abstract}

\section{Introduction}\label{S:Introduction}

The purpose of our paper is to analyze the local submonoids of enveloping semigroups. 
More precisely, we investigate submonoids of the form $eMe$ and $\overline{G_e^0}$,
where $M$ is an enveloping semigroup, $e\in M$ is an idempotent, and $G_e^0$ is 
the connected component of the identity of the stabilizer subgroup of $e$ in the unit group of $M$. 
Let $M_e$ denote monoid $\overline{G_e^0}$, which we call the {\em connected stabilizer of $e$ in $M$}.
We will characterize all idempotents $e$ for which 
$M_e$ has the property that $M_e \setminus G_e^0$ is a connected algebraic semigroup. 
To explain our results in more detail, and to motivate our discussion, we will introduce the enveloping semigroups 
in the historical order of their discovery.

Let $M$ be a reductive monoid with the group of invertible elements $G:=G(M)$ defined over an algebraically closed field $k^*$. 
Since $G$ is a connected reductive group, let us write it in the form $G \cong (G_0\times Z_G)/Z_0$,
where $G_0$ is the derived subgroup of $G$, $Z_G$ is the connected center of $G$, and $Z_0$ is the center of $G_0$. 
This data gives us an affine quotient morphism $\pi_M: M \to A_M$, where $A_M:=\textrm{Spec}\ ( k[M]^{G_0\times G_0})$.
The quotient $A_M$ is called the {\em abelianization} of $M$, and it is a $(G/G_0\times G/G_0)$-equivariant embedding of $G/G_0$.

Let $\mathcal{FM}(G_0)$ denote the set of reductive monoids $M$ such that 
\begin{enumerate}
%\item $S$ is a reductive monoid,
\item the derived subgroup of $G(M)$ is $G_0$, 
\item $\pi_M: M\to A_M$ is flat with reduced and irreducible fibers.
\end{enumerate}
Then, according to Vinberg~\cite{Vinberg1}, there is a unique normal reductive monoid with zero, denoted by 
$\textrm{Env}(G_0)$, in $\mathcal{FM}(G_0)$ such that 
\begin{itemize}
\item[-] the unit-group of $\textrm{Env}(G_0)$ is $(G_0\times T_0)/Z_0$, where $T_0$ is a maximal torus in $G_0$;
\item[-] there is an isomorphism $A_{\textrm{Env}(G_0)} \cong \mathbb{A}^l$, where $l = \dim T_0$; 
\item[-] every $M$ in $\mathcal{FM}(G_0)$ is a fiber product of the form 
\hbox{$M\cong A_M \times_{\mathbb{A}^l} \textrm{Env}(G_0)$}.
\end{itemize}
This remarkable semigroup, $\textrm{Env}(G_0)$, is called the {\em enveloping semigroup of $G_0$}.
In a related work~\cite{Vinberg2}, Vinberg showed also that the preimage $\pi_M^{-1}(0)$,
denoted by $\textrm{As}(G_0)$, is equal to the horospherical contraction of $G_0$ as a $G_0\times G_0$-variety. 
Following Vinberg's terminology, we will call $\textrm{As}(G_0)$ the {\em asymptotic semigroup} of $G_0$.

The results of Vinberg in~\cite{Vinberg1,Vinberg2}, which were originally obtained over an algebraically closed field of characteristic zero,
are shown to hold true in positive characteristic by Rittatore in~\cite{Rittatore:Thesis,Rittatore2001}.
Rittatore's approach, which uses the theory of spherical varieties, has been further generalized by Alexeev and Brion in~\cite{AB1} to 
a class of spherical varieties that they called the ``reductive varieties''. 
Since these objects are beyond the scope of our work, we will not discuss them; 
however, we anticipate extensions of some our results in their setting.
To explain our progress, and to put our work in the right perspective, first, we will briefly mention some observations due to Renner.

In his memoir~\cite{Renner}, Renner translated the aforementioned results of Vinberg and Rittatore to the 
language of idempotents. In particular, in the proof of~\cite[Theorem 6.18]{Renner}, which is about the ``type-map'' 
for enveloping semigroups, Renner introduced certain J-coirreducible monoids
(the precise definition of the type map will be given in the sequel).
One of these monoids fills the gap of the asymptotic semigroup, 
while the others entertain similar roles for the ``degenerate asymptotic semigroups.''
Here, by the filling the gap, we mean the adjoining of the group $k^*\cdot G_0$ to $\textrm{As}(G_0)$ 
so that $k^*\cdot G_0\sqcup \textrm{As}(G_0)$ becomes a normal semisimple monoid. 
Such an enlargement of $\textrm{As}(G_0)$ was already shown by Vinberg~\cite{Vinberg2} by an algebraic method, 
while Renner's approach is more geometric, which we will explain next. 
Let $(W,S)$ denote the Coxeter system, where $W$ is the Weyl group of $G_0$. Let $M$ denote $\textrm{Env}(G_0)$,
and let $G$ denote the unit group of $M$. 
Then for each nonempty subset $I\subset S$, there is a convergent one-parameter subgroup  
$\lambda_I : k^* \to A_M$ such that the limit $f:=f_I = \lim_{t\to 0} \lambda(t)$ is an idempotent in $A_M$. 
Let us define $M_I:= \pi^{-1}( \overline{\lambda_I(k^*)})$. 
In this notation, the following statements are observed by Renner in~\cite{Renner}:
\begin{enumerate}
\item[(1)] $M_I$ is a {\em J-coirreducible monoid}, that is to say, $M_I\setminus G(M_I)$ is a connected algebraic semigroup; 
\item[(2)] the ``cross-section lattice'' of $M$ is covered by the cross-section lattices of $M_I$'s,
\[
\Lambda(M) = \{1_G \} \bigsqcup_{I\subset S, I\neq \emptyset} \Lambda(M_I)\setminus \{1_{G(M_I)}\}.
\]
Here, by a cross-section lattice we mean a finite set of idempotents parametrizing the double-cosets of the unit group of the monoid.  
\item[(3)] If $I=S$, then $\textrm{As}(G_0) =  M_S \setminus  G(M_S)$.
\end{enumerate}
In summary, once a cross-section lattice for $M$ is fixed, one finds $2^{|S|-1}$  
J-coirreducible monoids which compartmentalize the idempotents of $M$ in such a way that 
one of the J-coirreducible monoids is equal to the asymptotic semigroup of $G_0$.

We are now ready to give an overview of our paper while describing its main results.
In Section~\ref{S:Preliminaries}, we review some fundamental results on the reductive monoids which we will use in the sequel. 
Starting from Section~\ref{S:Rank1}, we will focus on the enveloping semigroups $M=\text{Env}(G_0)$, 
where $G_0$ is a semisimple algebraic group such that the derived subgroup $(G_0,G_0)$ is a simple algebraic group.
The advantage of this restriction is that the Coxeter-Dynkin diagram of $G_0$ is connected. 
Let $\Lambda_1$ denote the minimal nonzero elements of the cross-section lattice of $M$.
In the first main result of the paper, we characterize the reductive monoids $eMe$ and $M_e$ ($e\in \Lambda_1$). 
We show that exactly $|S|$ of $eMe$'s with $e\in \Lambda_1$ are $G\times G$-stable. 
Furthermore, we show that such monoids are one dimensional, hence we obtain a description of the 
$G\times G$-stable curves in $M$. We determine the unit groups of $M_e$ for $e\in \Lambda_1$. 
In Section~\ref{S:Jcoirreducible}, we focus on the local monoids $eMe$ and $M_e$ where the rank of $e$ is $>1$. 
In particular, we characterize $e\in \Lambda$ for which the corresponding connected stabilizer monoid $M_e$ is 
J-coirreducible. In addition, if $G_0$ is one of the types $\text{A}_n,\text{B}_n,\text{C}_n, F_4$, or $G_2$,
then we characterize the ``J-linear'' connected stabilizer monoids (Theorem~\ref{T:Jcoirreducibles}).
Here, by a {\em J-linear monoid} we refer to a reductive monoid $M$ such that $\Lambda(M) \setminus \{ 0, 1\}$
has a unique minimal and a unique maximal element. 
In Theorem~\ref{T:unitgroups}, we determine the (unit) groups $C_G(e)=\{g\in G:\ ge=eg\},G(eMe)=eC_G(e)$, and $G(M_e) =G_e^0$ 
provided that $M_e$ is a maximal dimensional J-coirreducible monoid. It turns out that for such connected stabilizers, $G_e^0 = G_0$. 
The purpose of Section~\ref{S:Navel} is to show that, for $M=\textrm{Env}(G_0)$, 
there is always a unique idempotent $e$ in $\Lambda(M)$ such that both of the local monoids $M_e$ and $eMe$ are affine torus embeddings in $M$. 
Having such an idempotent is a rare phenomenon. 
We prove that in a J-coirreducible monoid of ``type $J$'' an idempotent $e\in \Lambda$ 
has the property that $C_G(e)$ is a torus embedding if and only if $J=\emptyset$. 
In the same section, we prove a similar result for the reductive monoids whose cross-section lattices have unique nonzero elements.
The purpose of Section~\ref{S:Atomic} is to show that the cross-section lattice of $M$ is generated by its rank 1 elements. 
Equivalently (and more precisely), we show that $\Lambda$ is an atomic lattice. 
In Section~\ref{S:Generatingfunction}, we present a combinatorial result, Theorem~\ref{T:numbers}, 
which gives a count of the number $G\times G$-orbits in the enveloping semigroup $\textrm{Env}(\textrm{SL}_n)$. 
More precisely, it states that the generating series of the number $d_n$, $n=0,1,2,\dots$ of $G\times G$-orbits 
in the enveloping monoid of $\textrm{SL}_n$ is given by 
\begin{align*}
\sum_{n\geq 0} d_n x^n  
%=\frac{x  }{(1-2x)(1-5x + 6x^2 -4x^3)}  + \frac{1}{1-2x}
= \frac{ 1-2x+2x^2   }{(1-5x + 6x^2 -4x^3)}   = 1+ 3x + 11x^2 + 41x^3 + 151x^4 + 553 x^5 + O(x^6).
\end{align*}

Before we finish our introduction, we want to mention a related work of ours. 
In~\cite{Can:Dual}, we investigate the nilpotent varieties and various partial orders on the asymptotic semigroups. 
In particular, we show that the nilpotent variety of $\textrm{As}(G_0)$ is an equidimensional variety, and we determine its Putcha poset.
We anticipate that some of our results from~\cite{Can:Dual} will have natural extensions to all J-coirreducible monoids. 
Also, by using the results of~\cite{Wedhorn, Huruguen}, and~\cite{Can:Mobile}, 
we can extend all results of this paper (and most of the results of~\cite{Can:Dual}) to the setting of the reductive monoid $k$-schemes,
where $k$ is any perfect field. We plan to come back to this generalization in a future paper. 
\\

\textbf{Acknowledgements.} 
We are grateful to the referees for their comments which improved the quality of our paper. 
In an earlier version of this article, we claimed that 
an enveloping semigroup $M$ is rationally smooth if and only if $M$ is of type A. 
A referee pointed to us that, even in type A, $M$ cannot be rationally smooth. 
We are thankful to the referee for the precise explanation, 
which has led us to look closer at the local structure of the enveloping monoids.
We thank Lex Renner for his guidance and help, not just for this project. 
Finally, we acknowledge that this research was partially supported by a grant from the Louisiana Board of Regents.

\section{Preliminaries}\label{S:Preliminaries}

In this article, we will focus on affine algebraic monoids defined over an algebraically closed field. 
We will follow the common terminology as set forward in the textbooks~\cite{Putcha} and~\cite{Renner}. 
In particular, we fix the following notation:
$$
\begin{array}{lcl}
M &: & \text{ a reductive monoid with zero;}\\
G &: & \text{ the unit group of $M$ (so, $G$ is a connected reductive group);}\\
T &: & \text{ a maximal torus in $G$;}\\
G_0 & : & \text{ the semisimple part of $G$;}\\
T_0 & : & \text{ a maximal torus in $G_0$ such that $T_0= T\cap G_0$;}\\
W &: & \text{ the Weyl group of $(G,T)$ (equivalently, of $(G_0,T_0)$);}\\
B &: & \text{ a Borel subgroup of $G$ such that $T\subset B$;}\\
\ell_W & : & \text{ the length function defined by $\ell_W(w):=\dim BwB - \dim B$ for $w\in W$;}\\
S &: & \text{ the set of Coxeter generators of $W$ determined by $B$;}\\
\overline{T} & : & \text{ the Zariski closure of $T$ in $M$;}\\
E(M), E(\overline{T}) &: & \text{ the sets of idempotents of $M$ and $\overline{T}$, respectively;}\\
\Lambda & : & \text{ the cross-section lattice determined by $(M,G,B,T)$;}\\
\lambda & : & \text{ the type-map $\lambda : \Lambda \to 2^S$ (the power set of $S$);}\\
R &:& \text{ the Renner monoid of $(M,T)$;}\\
\leq & : & \text{ the Bruhat-Chevalley-Renner order on $R$}.
\end{array}
$$

For the convenience of the reader, we will review the definitions of some of these objects. 
First of all, the order $\leq$ on $R$ is defined by 
\begin{equation}\label{E:BRordering}
\sigma \leq \tau\ \hspace{.51cm} \mbox{if and only if}\hspace{.51cm} 
B\sigma B \subseteq \overline{B\tau B}, \ \text{ where $\sigma, \tau \in R$.}
\end{equation}
The induced poset structure on $W$, which is induced from $R$ is the same as the well known Bruhat-Chevalley poset structure on $W$. 
There is a canonical partial order $\leq$ on the set of idempotents $E(M)$ of $M$ 
(hence, on the set of idempotents $E(\overline{T})$ of $\overline{T}$) defined by  
\begin{equation}\label{E:crossorder}
e\leq f \hspace{.51cm} \mbox{if and only if} \hspace{.51cm}  ef=e=fe.
\end{equation}
The set of idempotents of $\overline{T}$ is invariant under the conjugation action of $W$.
A subset $\Lambda \subseteq E(\overline{T})$ is called a {\em cross-section lattice}  
if $\Lambda$ is a set of representatives for the $W$-orbits on $E(\overline{T})$ and the bijection 
$\Lambda \rightarrow G \backslash M / G$ defined by $e \mapsto GeG$ is order preserving.

The {\em right centralizer of $\Lambda$ in $G$}, denoted by $C_G^r(\Lambda)$, is the subgroup 
\[
C_G^r(\Lambda) = \{ g\in G:\ ge = ege \text{ for all } e\in \Lambda\}.
\]	 
Assuming that $M$ has a zero, for all Borel subgroups of $G$ containing $T$ 
the set $\Lambda(B) = \{ e\in E(\overline{T}):\ Be=eBe\}$ is a cross-section lattice 
with $B= C_G^r(\Lambda)$, and for any cross-section lattice $\Lambda$, the 
right centralizer $C_G^r(\Lambda)$ is a Borel subgroup containing $T$ with 
$\Lambda = \Lambda ( C_G^r(\Lambda))$, see~\cite[Theorem 9.10]{Putcha}.

The decomposition of $M$ into $G\times G$-orbits has a counterpart in the Renner monoid, 
\[
R = \bigsqcup_{e\in \Lambda} WeW.
\]
The partial order (\ref{E:crossorder}) on $\Lambda$ agrees with the order induced from Bruhat-Chevalley-Renner order (\ref{E:BRordering}).
A fundamental result of Putcha asserts that (for any connected monoid $M$ with a zero) 
$E(\overline{T})$ is a relatively complemented lattice, anti-isomorphic to a face lattice of a convex polytope,
see~\cite[Theorem 6.20]{Putcha}. 
Let $\Lambda$ be a cross-section lattice in $E(\overline{T})$.  
The Weyl group of $T$ (relative to $B= C_G^r(\Lambda)$) acts on $E(\overline{T})$, and furthermore,
we have 
\[
E(\overline{T})= \bigcup_{w\in W} w \Lambda w^{-1}.
\]

The cross-section lattice with the order (\ref{E:crossorder}) is a graded poset; 
the rank function is given by 
\[
\text{rank}(e) = \dim T e \qquad (e\in \Lambda).
\]
A reductive monoid with zero is called {\em J-irreducible monoid} if it has a unique nonzero minimal $G\times G$-orbit. 
Equivalently, if the cross-section lattice of $M$ has a unique nonzero idempotent. 
Let us call $M\setminus G$ the {\em boundary of $M$}. 
We call $M$ a {\em J-coirreducible monoid} if the boundary of $M$ is the closure of a 
single $G\times G$-orbit. 
In other words, 
if $M$ is J-irreducible, then $1= | \{ e\in \Lambda :\ \text{rank}(e)=1 \} |$; 
if $M$ is J-coirreducible, then $1= | \{ e\in \Lambda :\ \text{corank}(e)=1 \} |$.
For more on these monoids, see~\cite[Chapter 14]{Putcha}.

Let $w$ be an element in $W$. Then $\ell_W(w)$ is equal to the minimal number of 
simple reflections $s_{i_1},\dots, s_{i_r}$ from $S$ with $w=s_{i_1}\cdots s_{i_r}$.
A subgroup that is generated by a subset $I\subset S$ will be denoted by $W_I$ and it will be called a 
{\it parabolic subgroup of $W$}. 
For $I\subseteq S$, we will denote by $D_I$ the following set:
\begin{align}
D_I:= \{ x\in W:\ \ell_W(xw ) = \ell_W(x) + \ell_W(w) \text{ for all } w\in W_I \}.
\end{align}

The {\em type-map}, $\lambda : \Lambda \to 2^S$, is defined by $\lambda(e) := \{s\in S:\ s  e= e s  \}$ for $e \in \Lambda$. 
The containment ordering between $G\times G$-orbit closures in $M$ 
is transferred via $\lambda$ to a sublattice of the Boolean lattice on $S$. 
Associated with $\lambda(e)$ are the following sets:
$\lambda_*(e) := \cap_{f\leq e} \lambda (f)$ and $\lambda^*(e):= \cap_{f\geq e} \lambda(f)$.
We define the subgroups 
\[
W(e):= W_{\lambda(e)},\qquad W_*(e) := W_{\lambda_*(e)},\qquad W^*(e) := W_{\lambda^*(e)}.
\]
Then we have 
\begin{enumerate}
\item $W(e)= \{ a\in W:\ ae=ea \}$,
\item $W^*(e)= \cap_{f\geq e} W(f)$,
\item $W_*(e)= \cap_{f\leq e} W(f) = \{ a\in W:\ ae= ea = e\}$.
\end{enumerate}
We know from~\cite[Chapter 10]{Putcha} that 
$W(e),W^*(e)$, and $W_*(e)$ are parabolic subgroups of $W$, and furthermore, we know that 
$W(e) \cong W^*(e) \times W_*(e)$. 
If $W(e) = W_I$ and $W_*(e) = W_K$ for some subsets $I,K\subset S$, then we define 
$D(e) := D_I$ and $D_*(e):= D_K$. 
\vspace{.25cm}

\textbf{Theorem/Definition (Pennell-Putcha-Renner):}
For every $x\in WeW$ there exist elements $a\in D_*(e), b\in D(e)$, which are uniquely determined by $x$, such that 
\begin{align}\label{A:std}
x= a e b^{-1}.
\end{align}
The decomposition of $x$ in (\ref{A:std}) will be called the {\em standard form of $x$}.
Let $e, f$ be two elements from $\Lambda$. It is proven in~\cite{PPR97} that 
if $x= a eb^{-1}$ and $y= cf d^{-1}$ are two elements in standard form in $R$, then 
\begin{align}\label{A:BCR}
x \leq y \iff e\leq f,\ a \leq cw,\ w^{-1} d^{-1} \leq b^{-1} \qquad \text{for some $w\in W(f)W(e)$.}
\end{align}
We will occasionally write $D(e)^{-1}$ to denote the set $\{ b^{-1} :\ b\in D(e)\}$.

Let $e$ be an idempotent from $M$. 
The {\em right centralizer of $e$}, denoted by $C_G^r(e)$, is the subgroup 
\[
C_G^r(e) := \{ g\in G:\ ge = ege\}.
\]
The {\em left centralizer of $e$}, denoted by $C_G^l(e)$, is defined similarly. 
The left and the right centralizers of $e$ are parabolic subgroups that are opposite to each other. 
In particular, their intersection, $C_G(e):=C_G^r(e) \cap C_G^l(e)$, is a common Levi subgroup.

We now introduce our ``local'' monoids. 
Let $e$ be an idempotent from $M$. 
The stabilizer of $e$ in $G$ is defined as $G_e:= \{g\in G:\ ge = eg =e \}$. 
This is group is not necessarily connected. 
The {\em connected stabilizer of $e$ in $M$} is the reductive monoid defined by 
\begin{align}\label{A:M_e}
M_e := \overline{G_e^0} = \overline{ \{g\in G:\ ge = eg =e \}^0}.
\end{align}
Thus, the group of invertible elements of $M_e$ is given by the connected component of identity $G_e^0\subset G_e$.
It is easy to see that $e$ is the zero element in $M_e$. It is also easy to verify that 
the set of idempotents of $M_e$ is $E(M_e) = \{ f\in E(M) :\ f\geq e \}$, see~\cite[Theorem 6.7]{Putcha}.
In fact, as shown in~\cite[Lemma 10.16]{Putcha}, the cross-section lattice of $M_e$ is given by 
\begin{align}\label{A:CSofM_e}
\Lambda (M_e) = \{ f\in \Lambda (M):\ f\geq e\}.
\end{align}
Another closely related subsemigroup is given by $eMe= \{ x\in M:\ exe=x\}$.
The cross-section lattice of $eMe$ is given by $e\Lambda = \{ f\in \Lambda :\ f\leq e\}$. 
The following lemma will be used several times in the sequel. Its proof is recorded in~\cite[Lemma 1.2.2]{Brion08}.
\begin{Lemma}\label{L:M_e}
Let $e$ be an idempotent in $M$. Then $C_G(e)$ and $G_e$ are reductive groups. 
Furthermore, $G_e$ is a normal subgroup of $C_G(e)$; there is an exact sequence of algebraic groups
\hbox{$1 \to G_e \to C_G(e) \to G(eMe) \to 1$}. 
The normalizer of $G_e$ in $G$ is equal to $C_G(e)$. 
\end{Lemma}

The proof of the next lemma can be found in~\cite[Proposition 10.9]{Putcha}.
\begin{Lemma}\label{L:WofeMe}
Let $e$ be an idempotent in $M$.
Then 
\begin{enumerate}
\item $C_W(e) = W(C_G(e))$;
\item $W(eMe) = eC_W(e) = \{ e w :\ w\in C_W(e)\}$;
\item $C_W(e) = W(eMe) \times W(M_e)$.
\end{enumerate}
\end{Lemma}
Let us now specialize to an idempotent from $\Lambda$. Then Lemma~\ref{L:WofeMe} 
shows that $W(C_G(e))= W(e)$ in our previous notation. 
Recall that the type-map of $\Lambda$ is equal to $\lambda = \lambda^* \sqcup \lambda_*$. 
In this notation, the restriction of $\lambda^*$ to $\Lambda(eMe)$ agrees with the 
$\lambda^*$ of the cross-section lattice $\Lambda(eMe)$ and the restriction of $\lambda_*$ to $\Lambda(M_e)$
agrees with the $\lambda_*$ of $\Lambda(M_e)$.

\section{Rank 1 Elements}\label{S:Rank1}

In this section, we analyze some local monoids in the enveloping semigroups. 
For easing our notation, we denote $\textrm{Env}(G_0)$ by $M$. 
\begin{Notation}
Let $\varSigma$ denote the Coxeter-Dynkin diagram of $(G,B,T)$. For every subset $I\subset S$, 
we have the corresponding subdiagram $\varSigma_I$ in $\varSigma$. 
By abusing notation, we will not distinguish between $I$ and $\varSigma_I$. In particular, 
the subsets of $I$ that correspond to the connected components of $\varSigma_I$
will be called the {\em connected components of $I$}. 
\end{Notation}

In~\cite[Section 6]{Renner}, by using the language that himself and Putcha developed, 
Renner reviews Vinberg's work on enveloping semigroups. 
In~\cite[Theorem 6.18]{Renner}, Renner describes
\begin{enumerate}
\item the cross-section lattice of $M$,
\begin{align}\label{A:descriptionofLambda}
\Lambda(M) = \{ e_{I,J}:\ I, J \subseteq S, \text{ and no component of $J$ is contained in $S\setminus I$}\};
\end{align}
\item the type-map of $M$, $\lambda (e_{I,J}) = \lambda_*(e_{I,J})\sqcup \lambda^*(e_{I,J})$, where 
\begin{align*}
\lambda_*(e_{I,J}) = J\ \text{ and } \ 
\lambda^*(e_{I,J}) = \{ s \in S \setminus I :\ \text{ $s s' =s' s$
for all $s' \in J$}\}.
\end{align*}
\end{enumerate}
In this notation, the natural partial order on the cross-section lattice is given by 
\begin{align}\label{A:essentialorder}
e_{I_1,J_1} \leq e_{I_2,J_2} \iff I_1 \supseteq I_2 \ \text{ and } J_1 \supseteq J_2,
\end{align}
where $I_i,J_i \subseteq S$ for $i\in \{1,2\}$.

\begin{Remark}
The empty set is not assumed to be a connected component of $J$, therefore, $\Lambda$ contains all idempotents $e$ 
of the form $e=e_{I,\emptyset}$, $I\subseteq S$. 
It follows from (\ref{A:essentialorder}) that the minimum and the maximum elements in $\Lambda$ 
are given by $e_{S,S}$ and $e_{\emptyset,\emptyset}$, respectively. 
It is easy to check that the height of $\Lambda$, as a graded poset, is equal to $2|S|$. 
\end{Remark}

\begin{Assumption}
In the rest of this section, we assume that the derived subgroup $(G_0,G_0)$ is a simple algebraic group.
In particular, the Dynkin diagram of $G_0$ is connected. 
\end{Assumption}

\begin{Lemma}\label{L:minimal corank1}
Let $\Lambda = \Lambda(M)$ denote the cross-section lattice of an enveloping monoid.
Then the smallest element of $\Lambda$ is given by $e_{S,S}$.
Furthermore, an element $e_{I,J}$ covers $e_{S,S}$ if and only if $|I|+ |J| = 2|S| -1$.
\end{Lemma}

\begin{proof}
Both of these claims follow from the descriptions of $\Lambda$ and its ordering,
see (\ref{A:descriptionofLambda}) and (\ref{A:essentialorder}).
\end{proof}

Let $R$ denote the Renner monoid of $M$, and let $\Lambda$ denote the cross-section lattice of $M$. 
We define the subsets $\Lambda_1$ and $\Lambda^1$ in $\Lambda$ as follows:
\begin{align*}
\Lambda_1 := \{ e\in \Lambda :\ \text{rank}(e)=1 \}\qquad \text{and}\qquad 
\Lambda^1 := \{ e\in \Lambda :\ \text{corank}(e)=1 \}.
\end{align*}
It is easy to check that $\Lambda_1$ has $2|S|$ elements, and $\Lambda^1$ has $|S|$ elements.

\begin{Example}\label{E:Poset}

The partial order (\ref{A:essentialorder}) for $M=\textrm{Env}(\textrm{SL}_3)$ is depicted in Figure~\ref{F:1}. 
It gives the inclusion order between the closures of all $G\times G$-orbits in $M$, 
where $G= (\textrm{SL}_3\times T_0)/Z_{ \textrm{SL}_3}$.
Here, $T_0$ is the maximal torus in $ \textrm{SL}_3$, and $Z_{ \textrm{SL}_3}\cong \Z/3$.
\begin{figure}[htp]
\centering
\begin{tikzpicture}[scale=.35]
\begin{scope}
\node at (0,-10) (t1) {$e_{\{1,2\},\{1,2\}}$};
\node at (-7.5,-5) (t2) {$e_{\{1\},\{1,2\}}$};
\node at (-2.5,-5) (t3) {$e_{\{1,2\},\{1\}}$};
\node at (2.5,-5) (t4) {$e_{\{1,2\},\{2\}}$};
\node at (7.5,-5) (t5) {$e_{\{2\},\{1,2\}}$};
\node at (-5,0) (t6) {$e_{\{1\},\{1\}}$};
\node at (0,0) (t7) {$e_{\{1,2\},\emptyset}$};
\node at (5,0) (t8) {$e_{\{2\},\{2\}}$};
\node at (-2.5,5) (t9) {$e_{\{1\},\emptyset}$};
\node at (2.5,5) (t10) {$e_{\{2\},\emptyset}$};
\node at (0,10) (t11) {$e_{\emptyset,\emptyset}$};

\draw[-, thick] (t1) to (t2); 
\draw[-, thick] (t1) to (t3);
\draw[-, thick] (t1) to (t4);
\draw[-, thick] (t1) to (t5);
\draw[-, thick] (t2) to (t6);
\draw[-, thick] (t3) to (t6);
\draw[-, thick] (t3) to (t7);
\draw[-, thick] (t4) to (t7);
\draw[-, thick] (t4) to (t8);
\draw[-, thick] (t5) to (t8);
\draw[-, thick] (t6) to (t9);
\draw[-,thick] (t7) to (t9);
\draw[-,thick] (t7) to (t10);
\draw[-, thick] (t8) to (t10);
\draw[-, thick] (t9) to (t11);
\draw[-, thick] (t10) to (t11);
\end{scope}

\end{tikzpicture}
\caption{The cross-section lattice of $\textrm{Env}(\textrm{SL}_3)$.}
\label{F:1}
\end{figure}
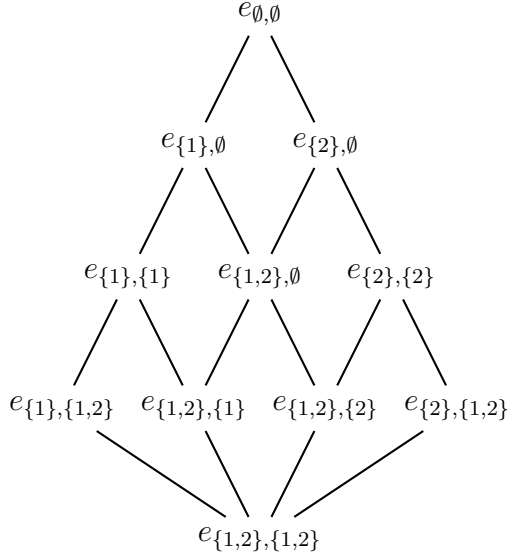

\end{Example}

\begin{Proposition}\label{P:W of eMe}
Let $e_{I,J}$ be an element of $\Lambda_1$. 
Then $W(eMe) = \{1 \}$.
\end{Proposition}

\begin{proof}
By Lemma~\ref{L:minimal corank1}, we have
$|I| + |J| = 2 |S| -1$. 
It follows that there is a unique $s_i\in S$ such that either $I = S\setminus \{s_i\}$ or $J=S\setminus \{s_i\}$ hold. 
Since the type map of $M$ is given by 
$\lambda (e_{I,J}) = \lambda_*(e_{I,J})\sqcup \lambda^*(e_{I,J})$, where 
\begin{align*}
\lambda_*(e_{I,J}) = \{ s_i :\ i\in J\} \ \text{ and } \ 
\lambda^*(e_{I,J}) = \{ s_{i} :\ s_i \in S \setminus I\ \text{ and $s_{i} s_{j} =s_{j} s_{i}$
for all $s_j \in J$}\},
\end{align*}
and since the Dynkin diagram of $G_0$ is connected, 
we see that $\lambda^* (e_{I,J})=\emptyset$ and $\lambda (e_{I,J}) = \lambda_*(e_{I,J})=J$. 
In particular, we have 
\begin{align}\label{A:forrank1}
W(e_{I,J}) \cong W_{\lambda^*(e_{I,J})} \times W_{\lambda_*(e_{I,J})} = \{1 \} \times W_{\lambda_*(e_{I,J})} \cong  W_{\lambda_*(e_{I,J})}.
\end{align}
Recall from Lemma~\ref{L:WofeMe} that $W( eMe ) = eC_W(e)$ for any idempotent $e$ in $M$.
In our case, this means that $C_W(e_{I,J}) = W(e_{I,J})$ and that $W(e_{I,J}Me_{I,J})  = e_{I,J}W(e_{I,J})$.
But (\ref{A:forrank1}) shows that $W(e_{I,J}) = W_*(e_{I,J})= \{ w\in W :\ w e_{I,J} = e_{I,J} w = e_{I,J} \}$.
Therefore, 
\[
W(e_{I,J}Me_{I,J}) = \{ e_{I,J} w:\ w\in W_*(e_{I,J}) \} = \{ e \}.
\]
Clearly, $\{e\}$ is the trivial group, hence, the proof is finished. 
\end{proof}

We proceed with an important corollary of Proposition~\ref{P:W of eMe}.
\begin{Corollary}\label{C2:W of eWe}
Let $e$ be an idempotent from $\Lambda_1$. 
If $e=e_{I,J}$ ($I,J\subseteq S$), then the local monoid $e M e$ is isomorphic to the affine line $\mathbb{A}^1$. 
Moreover, we have the isomorphism $k^* \cong C_G(e) / G_e^0$. 
\end{Corollary}
\begin{proof}
The cross-section lattice of $eMe$ is given by $\Lambda( eMe) = \{ 0, e_{I,J} \}$. 
In other words, $eMe$ is a J-irreducible monoid. It follows from Proposition~\ref{P:W of eMe} that 
the unit group of $eMe$ is a torus. But J-irreducible monoids are semi-simple, hence their centers are one dimensional,
see~\cite[Lemma 8.3.2]{Renner85}. It follows that the unit group of $eMe$ is isomorphic to $k^*$, hence, 
$eMe= \mathbb{A}^1$.

The second claim follows from Lemma~\ref{L:M_e}.
\end{proof}

In the next proposition, $D_J^{op}$ stands for the opposite of the Bruhat-Chevalley order on the minimal coset representatives
$D_J$ of $W_J$ in $W$. 

\begin{Proposition}\label{P:W of eWe}
Let $e$ be an idempotent from $\Lambda_1$. 
If $e=e_{I,J}$ ($I,J\subseteq S$), then we have
\begin{align*}
WeW \cong  
\begin{cases}
\{ 1\} & \text{ if } J=S; \\
D_J\times D_J^{op} & \text{ if } J= S\setminus \{s\}.
\end{cases}
\end{align*}
In particular, the dimension of the orbit $GeG$ is given by 
\begin{align*}
\dim {GeG} = 
\begin{cases}
1 & \text{ if } J=S; \\
2\frac{|W|}{ |W_J|} +1  & \text{ if } J= S\setminus \{s\}.
\end{cases}
\end{align*}
\end{Proposition}

\begin{proof}

Since $e$ is a minimal idempotent in $\Lambda \setminus \{ 0 \}$, 
we have $\overline{GeG} = \{ 0 \} \cup \bigsqcup_{\sigma \in WeW } B \sigma B$.
For any $f\in \Lambda$, the set of $B\times B$-orbit closures in $\overline{GfG}$ 
form a graded poset ranked by the dimension function.
Therefore, if we prove our first claim, our second claim will follow at once by computing the length
of the maximal element in the Bruhat poset.

To this end, we will use the standard form of the elements in $WeW$. 
It is easy to verify from (\ref{A:forrank1}) that 
\begin{align}\label{A:forrank1'}
D_*(e)= D(e) = 
\begin{cases}
\{ id \} & \text{ if } J=S; \\
D_J & \text{ if } J= S\setminus \{s\}.
\end{cases}
\end{align}
Therefore, in the first case, we have $WeW = \{e\}$. In the second case, we have 
$WeW = \{ w e w' :\ w\in D_J \}$. 
It is not difficult to show, by using (\ref{A:BCR}), that 
if $\sigma = w_1 e w_1', \tau = w_2 e w_2'$ are two elements from $WeW$, then 
\[
\sigma \leq \tau \iff w_1 \leq w_2 \ \text{ and } \ w_1' \geq w_2'.
\]
Therefore, the inclusion posets of $B\times B$-orbit closures in $WeW$ is isomorphic to 
$D_J \times D_J^{op}$ as claimed. 
\begin{comment}
are as in Figure~\ref{F:minimal}. 
\begin{figure}[htp]
\centering
\begin{tikzpicture}[scale=.35]

\begin{scope}[xshift=-5cm]
\node at (0,-10) (t1) {$0$};
\node at (0,-5) (t2) {$e$};
\draw[-, thick] (t1) to (t2); 
\end{scope}

\begin{scope}[xshift=5cm]
\node at (0,-10) (t1) {$0$};
\node at (0,-5) (t2) {$es$};
\node at (2.5,0) (t3) {$e$};
\node at (-2.5,0) (t4) {$ses$};
\node at (0,5) (t5) {$se$};

\draw[-, thick] (t1) to (t2); 
\draw[-, thick] (t2) to (t3);
\draw[-, thick] (t2) to (t4);
\draw[-, thick] (t3) to (t5);
\draw[-, thick] (t4) to (t5);
\end{scope}
\end{tikzpicture}
\caption{The Bruhat-Chevalley-Renner order on the minimal nonzero $G\times G$-orbit closures.}
\label{F:minimal}
\end{figure}
Clearly, the poset on the left hand side is of height 1, and the other is of height 3. 
\end{comment}
This finishes the proof.
\end{proof}

\begin{Remark}
Note that Proposition~\ref{P:W of eWe} implies that there are exactly $|S|$ $G\times G$-stable curves in $M$ and 
they all contain 0. Of course, these curves are $B\times B$-stable as well. 
Going through this vein, we see that the total number of $B\times B$-stable curves in $M$ is given by $2|S|$. 
More generally, let $R$ denote the Renner monoid of $M$, and let $R_1$ denote its subposet $\bigsqcup_{e\in \Lambda_1} WeW$. 
The following formula is easily proven by using Proposition~\ref{P:W of eWe}:
\[
| R_1 | = |S| + \sum_{s\in S} \left(\frac{ |W| }{|W_{S\setminus \{s\}} |}\right)^2.
\]
Here, the first term, $|S|$, is equal to the number of the posets $We_{I,J}W$ with $|J| +1 = |I| = |S|$, and 
the summation gives the total number 
of elements in all posets of the form $We_{I,J} W$ with $J=S$ and $| I | = |S|-1$. 
Each of these subposets has a unique minimal element corresponding to a $B\times B$-stable curve. 
\end{Remark}

\vspace{.25cm}

Next, we determine the groups $G_e^0$ and $C_G(e)$ for $e\in \Lambda_1$. 

\begin{Corollary}\label{C3:W of eWe}
Let $e$ be an idempotent from $\Lambda_1$. If $e=e_{I,J}$ 
with $I,J\subseteq S$, then the following statements hold true:
\begin{enumerate}
\item if $J=S$, then $C_G(e) \cong G_0\times T_0'$, where $T_0'$ is a codimension one subtorus in $T_0$;
\item if $I=S$, then $C_G(e) \cong L \times T_0$, where $L$ is a Levi subgroup of a maximal parabolic subgroup of $G_0$.
\end{enumerate}
Furthermore, in the former case, we have $G_e^0 \cong G_0 \times T_0''$, where $T_0''$ is a codimension one subtorus in $T_0'$,
and in the latter case, we have $G_e^0 \cong L\times T_0'$, where $T_0'$ is a codimension one subtorus in $T_0$. 
\end{Corollary}

\begin{proof}
By the proof of Proposition~\ref{P:W of eMe} we know that  
for $I=S$, $W(C_G(e))$ is the maximal parabolic subgroup $W_J$ in $W$, and for $J=S$, 
it is isomorphic to $W$. In other words, if $I=S$, then $C_G(e)$ contains a copy of the Levi subgroup of $G_0$ 
that is determined by $J$, and if $J=S$, then $C_G(e)$ contains a copy of $G_0$. 
Now, our first two claims follow from these facts and Corollary~\ref{C2:W of eWe}.

For our second claim, we argue in a similar way; recall from the proof of Proposition~\ref{P:W of eMe} that 
the Weyl groups of the reductive groups $C_G(e)$ and $G_e^0$ are isomorphic. 
In other words, the structures of the derived subgroups of both of these groups are isomorphic. 
But since $G_e^0$ is a (normal) subgroup of $C_G(e)$ and the quotient is a one dimensional torus, the proof is finished. 
\end{proof}

\vspace{.25cm}

We finish this section with a remark on the automorphisms of $\Lambda$. 
\begin{Remark}
It is well-known that the only irreducible reduced root systems with nontrivial automorphisms are of types $\text{A}_n$, $\text{D}_n$, and $\text{E}_6$.
The automorphism groups of these root system are given by 
\begin{enumerate}
\item $\text{Aut}(\text{A}_n) = \Z/2$; 
\item $\text{Aut}(\text{D}_n) = \Z/2$ for $n > 4$; 
\item $\text{Aut}(\text{D}_4) = S_3$;
\item $\text{Aut}(\text{E}_6) = \Z/2$.
\end{enumerate} 
Any nontrivial automorphism of these root systems preserves the connected components as well as the containment relations between 
connected components of the sub-Coxeter-Dynkin diagrams. 
Thus, it gives a nontrivial automorphism of the cross-section lattice. 
%In particular each nontrivial automorphism gives an automorphism of the cross-section lattice. 
\end{Remark}

\section{The J-coirreducible Slices of $\textrm{Env}(G_0)$}\label{S:Jcoirreducible}

The purpose of this section is to determine the idempotents in $M=\textrm{Env}(G_0)$ for which the corresponding 
connected stabilizer monoid $M_e$ is J-coirreducible. 
We maintain our notation from Section~\ref{S:Rank1}. 
In particular, we assume that $G_0$ is a simple algebraic group.

Let $e= e_{I,J}$ be an element from $\Lambda$.
In Section~\ref{S:Rank1} we observed that if $J=S$ and $e\in \Lambda_1$, then $eMe$ is one dimensional,
and the unit group of $M_e$ is isomorphic to $G_0\times T_0'$, where $T_0'$ is a codimension one subtorus in $T_0$. 
We will present a generalization of this observation. First, we have a simple remark. 

\begin{Remark}\label{R:maximalelement}
Let $e_{I,J}$ be an idempotent from $\Lambda$. 
We claim that if $I=\emptyset$, then $e_{I,J}$ is the maximal element in $\Lambda$.
We already noted in Section~\ref{S:Rank1} that $e_{\emptyset,\emptyset}$ is the maximal element of $\Lambda$. 
By definition, $e_{I,J}$ is an element of $\Lambda$ if and only if no connected component of $J$ is entirely contained in $S\setminus I$. 
Therefore, if $I=\emptyset$, then there is no proper subset $J\subset S$ such that $e_{I,J} \in \Lambda$. This finishes the proof of our claim. 
\end{Remark}

\begin{Lemma}\label{L:prep1}
If the set $I$ is a singleton, then the connected stabilizer corresponding to an idempotent of the form $e=e_{I,J}$
($J\subseteq S$) is a J-coirreducible monoid. 
\end{Lemma}

\begin{proof}
Let $e_{I,J}$ be an idempotent from $\Lambda$. By Remark~\ref{R:maximalelement}, 
if $I=\emptyset$, then $e_{I,J}$ is the maximal element in $\Lambda$, so, we assume that $I\neq \emptyset$. 
Next, we observe the simple fact that, for every $I,J\subset S$, we have $e_{I,J} \leq e_{I,\emptyset}$. In particular, for $I=\{s\}$,
we have $e_{I,J} \leq e_{I,\emptyset}$. But $e_{I,\emptyset}$ is the unique element from $\Lambda^1$ that has this property. 
In other words, for any $J\subset S$, the maximum element $e_{\emptyset, \emptyset}$ covers a unique element 
of the upper interval $[e_{\{s\},J}, e_{\emptyset, \emptyset}]$. This finishes the proof.
\end{proof}

\begin{Remark}\label{R:descriptions}
The proof of the following fact follows from the descriptions of the cross-section lattices of the connected stabilizer monoids,
see~(\ref{A:CSofM_e}): 
If $e_{I,J} \leq e_{I',J'}$ are two idempotents from $\Lambda$ such that $M_{e_{I,J}}$ is J-coirreducible, then so is $M_{e_{I',J'}}$.
\end{Remark}

In the proof of the converse of Lemma~\ref{L:prep1},  
we will need the following notion.
\begin{Definition}
Let $K$ be a subset in $S$ viewed as a subdiagram in the Coxeter-Dynkin diagram of $S$.
An {\em end-node} in $K$ is an element $u\in K$ such that there exists $s\in S\setminus K$ with $su\neq us$.
\end{Definition}

\begin{Lemma}\label{L:prep2}
Let $e=e_{I,J}$ be an element from $\Lambda$. 
If $M_e$ ($e\in \Lambda$) is a J-coirreducible connected stabilizer in $M$,
then $I= \{s\}$ for some $s\in S$. 
\end{Lemma}
\begin{proof}
Let $e_{I,J}$ be as in the hypothesis. 
We now assume towards a contradiction that $I$ contains at least two different elements, $u,v\in I$ with $u\neq v$.
In this case,  if $J$ is the empty set, then $e_{\{u\},\emptyset} \geq e$ and $e_{\{v\},\emptyset} \geq e$;
this contradicts with our assumption on the ``J-coirreducibleness'' of $M_e$. 
With this observation and Remark~\ref{R:descriptions}, to finish the proof, it suffices to show that if $J$ is nonempty, 
then there is an element $u\in J$ such that $e_{I,J\setminus \{u\}} \in \Lambda$.

On one hand, if $J\cap I = \emptyset$, then we must have $J=\emptyset$. Therefore, we assume that $J\cap I \neq \emptyset$. 
On the other one hand, if $J$ is contained in $I$, then for every $u\in J$, we have the $e_{I,J\setminus \{u\}} \in \Lambda$.
Therefore, we assume that $J$ is not entirely contained in $I$. 
If $u$ is an end-node of $J$ such that $u\in J\cap (S\setminus I)$, then clearly no connected component of $J\setminus \{u\}$ 
is entirely contained in $S\setminus I$, hence, $e_{I, J\setminus \{u\}}\in \Lambda$. 
We now assume that all end-nodes of $J$ are contained in $I$. 
In this case, if $u\in J$ is an isolated connected component of $J$, then we still have $e_{I,J\setminus \{u\}} \in \Lambda$.
Therefore, we assume that if $K$ is a connected component of $J$, then all end-nodes of $K$ are contained in $I$. 
Let $u\in K$ be an end-node. Then since other end-nodes of $K$ are still contained in $I$, we see that $K\setminus \{u\}$
is not entirely contained in $S\setminus I$, therefore, $e_{I,J\setminus \{u\}} \in \Lambda$. 
We finished showing that under our assumptions there is always a node $u\in J$ such that removing $u$ from $J$ 
gives us another, bigger idempotent $e_{I,J} \leq e_{I,J\setminus \{u\}}$ in $\Lambda$. This finishes the proof.
\end{proof}

\begin{Theorem}\label{T:Jcoirreducibles}
Let $e:=e_{I,J}$ ($I,J\subset S$) be an idempotent from $\Lambda$. Then 
$M_e$ is J-coirreducible if and only if $I=\{s\}$ for some $s\in S$. 
Moreover, in this case, if we assume that the Coxeter-Dynkin diagram of $G_0$ is one of the types 
$\text{A}_n,\text{B}_n,\text{C}_n, F_4$, or $G_2$. then $M_e$ is J-irreducible if and only if $s$ is an end-node. 
\end{Theorem}

\begin{proof}
By combining Lemmas~\ref{L:prep1} and~\ref{L:prep2}, we obtain the proof of the first claim.  
For our second claim, we first note that, by our additional assumption on the Coxeter-Dynkin diagrams, 
any connected component of $J$ has exactly two end-nodes. 
We note also that since $e_{\{s\},J}\in \Lambda$, we have 
1) $s\in J$, 2) the connected component of $s$ in $J$ is equal to $J$.
Now, if $e_{I,J'}$ covers $e_{\{s\},J}$ and $I=\emptyset$, then by Remark~\ref{R:maximalelement} we know that 
$e_{I,J'}$ is the maximal element, so, our claim is true in this case.
If $I'=I$, then $J\setminus J'$ is a singleton, and furthermore, $s\in J'$ and the connected component of $s$ in $J'$ is equal to $J$.
But this means that $J'$ is obtained from $J$ by removing an end point. 
This finishes the proof. 
\end{proof}

In view of Remark~\ref{R:descriptions} and Theorem~\ref{T:Jcoirreducibles}, the following definition is meaningful.
\begin{Definition}
A J-coirreducible connected stabilizer $M_e$ in $M$ is called maximal if $e$ is of the form $e=e_{\{s\},S}$ for some $s\in S$. 
\end{Definition}

\begin{Remark}
The list of J-irreducibles connected stabilizers $M_e$ is not fully determined by Theorem~\ref{T:Jcoirreducibles}. 
It would be interesting to find a characterizations of the idempotents $e\in \Lambda$ such that $M_e$ is J-irreducible.
\end{Remark}

In our next result we determine explicitly the unit groups of maximal J-coirreducible connected stabilizers.

\begin{Theorem}\label{T:unitgroups}
Let $e$ be an idempotent from $\Lambda$. 
If $e$ is of the form $e=e_{\{s\},S}$ for some $s\in S$, 
then $C_G(e)$ is isomorphic to $G_0\times T_0'$, where $T_0'$ is a codimension one subtorus in $T_0$.
Furthermore, in this case, we have $C_G(e) / G_e^0 \cong (k^*)^{|S|-1}$.
\end{Theorem}

\begin{proof}
We argue as in the proof of Corollary~\ref{C3:W of eWe}.
Since $\lambda(e ) = \lambda_*(e) = S$, we see that $\lambda^*(e)=\emptyset$ and that the Weyl group of $C_G(e)$ is $W$.
Since $\lambda_*$ of $\Lambda(M)$ agrees with the $\lambda_*$ of $\Lambda(M_e)$, we see also that $W(G_e^0)=W$,
and therefore, $W(eC_G(e))=\{1\}$. It follows that both of the groups $C_G(e)$ and $G_e^0$ contain a copy of $G_0$,
and $eMe$ is a torus embedding. Recall that the cross-section lattice of $eMe$ is given by 
$\Lambda ( eMe)= \{ f\in \Lambda (M):\ f \leq e\}$. It is easy to verify that $e_{I,J}\leq e_{\{s\}, S}$ if and only if $J=S$, and 
$I$ is a subset of $S$ with $s\in I$. Therefore, $\Lambda(eMe)$ is isomorphic to the Boolean lattice on the set $S\setminus \{s\}$.
In particular, its height is $|S|-1$. This observation shows that, as a torus embedding, the dimension of $eMe$ is given by 
$\dim eMe = |S|-1$, hence, the group of units of $eMe$, that is $G(eMe)$, is a torus of dimension $|S|-1$. 
By Lemma~\ref{L:M_e}, this proves our last claim. 
To see the validity of our first claim, note that $G_0\subset G_e$, hence, once again by Lemma~\ref{L:M_e}, $G(eMe) \cap G_0 = \{1\}$. 
In particular, the maximal torus of $C_G(e)$ is at least $2|S|-1$ dimensional. Clearly, it cannot be of dimension $2|S|$, otherwise, 
we would have $C_G(e)=G$. From this we conclude that a maximal torus in $C_G(e)$ has dimension exactly $2|S|-1$, therefore, 
$C_G(e)$ is of the form $G_0\times T_0'$ with $T_0'\subset T_0$ a codimension one subtorus. This finishes the proof. 
\end{proof}

The proof of the following result now follows from Theorem~\ref{T:unitgroups} and Lemma~\ref{L:M_e}.

\begin{Corollary}
If $e$ is an idempotent as in Theorem~\ref{T:unitgroups}, then $G_e^0$ is isomorphic to $G_0$. 
\end{Corollary}

We proceed to determine the ``types'' of the J-coirreducible connected stabilizers in $M$. 
\begin{Definition}\label{D:type}
Let $M$ be a J-coirreducible monoid with the unique idempotent $e_0$ from $\Lambda^1(M)$. 
The {\em type of $M$} is the subset $I = \lambda(e_0)$ in $S$. 
We will denote the type of $M$ by $\text{type}(M)$. 
\end{Definition}

We return to our notational convention that $M=\text{Env}(G_0)$.
Let $M_e$ be a maximal J-coirreducible monoid in $M$. By Theorem~\ref{T:unitgroups}, we know that the unit group of $M_e$ 
is $G_0$, hence, its set of simple roots is given by $S$. 

\begin{Proposition}
Let $M_e$ be a maximal J-coirreducible monoid in $M$. 
If $e=e_{\{s\},S}$, then $\text{type}(M_e) = S\setminus \{s\}$. 
\end{Proposition}
\begin{proof}
Note that the unique maximal element in $\Lambda(M_e)\setminus \{e_{\emptyset,\emptyset}\}$ is given by $e_{\{s\},\emptyset}$.
The $\lambda^*$ of $M_e$ is given by the restriction of $\lambda^*$ of $\Lambda$. 
Since $\lambda^*(e_{I,J}) = \{ s \in S \setminus I :\ \text{ $s s' =s' s$ for all $s' \in J$}\}$, 
we see that $\lambda^*(e_{\{s\},\emptyset}) = S\setminus \{s\}$. On the other hand, we have $\lambda = \lambda^* \sqcup \lambda_*$.
But we cannot have $\lambda_*( e_{\{s\},\emptyset} ) \neq \emptyset$, otherwise, $\lambda(e_{\{s\},\emptyset}) = S$;
the full set of simple roots cannot be the type of a J-coirreducible monoid. Hence, $\lambda(e_{\{s\},\emptyset}) = \lambda^*(e_{\{s\},\emptyset})=S\setminus \{s\}$. This finishes the proof.  
\end{proof}

\section{The Navel}\label{S:Navel}

We will now discuss another extreme case. 
Let us call an idempotent in a cross-section lattice $e\in \Lambda$ a {\em navel of $\Lambda$} if $\lambda(e) = \emptyset$. 
If $e$ is a navel of $\Lambda$, then we have $\lambda^*(e) = \lambda_*(e) = \emptyset$. 
Clearly, the converse of this statement is true as well; if $\lambda^*(e) = \lambda_*(e) = \emptyset$, then $e$ is a navel. 
Not every reductive monoid has a navel. 
To give an example, let us consider the monoid of $n\times n$ matrices $M=\textrm{Mat}_n$. 
Let $e_r$ ($r\in \{1,\dots, n\}$) denote the matrix 
\[
e_r := \mathbf{I}_r \oplus \mathbf{0}_{n-r},
\] 
where $\mathbf{I}_r$ is the neutral element in $\textrm{Mat}_r$, and $\mathbf{0}_{n-r}$ is the zero element of $\textrm{Mat}_{n-r}$. 
Then $\Lambda = \{ 0,e_1,\dots, e_n\}$ is a cross-section lattice for $M$.
It is easy to verify that, for every $e_r \in \Lambda$, either $\lambda_*(e_r) \neq \emptyset$, or $\lambda^*(e_r)\neq \emptyset$. 
Thus, $\textrm{Mat}_n$ has no navel. 
In the sequel we will present a generalization of this observation. 

\begin{Lemma}\label{L:Navel}
If $e$ is a navel of $\Lambda$ in a reductive monoid $M$, then the reductive groups $C_G(e), G_e^0$, and $G(eMe)$
are tori. In particular, both of the local monoids $M_e$ and $eMe$ are affine torus embeddings. 
\end{Lemma}

\begin{proof}
Since $W(e) = W_*(e) = W^*(e) = \{1\}$, the reductive groups have no unipotent components, therefore, 
$C_G(e)$, $G_e^0$, and $G(eMe)$ are tori.
\end{proof}

\begin{Definition}\label{D:type co}
Let $M$ be a J-irreducible monoid with the unique minimal nonzero idempotent $e_0$ from $\Lambda_1(M)$. 
The {\em type of $M$} is the subset $I = \lambda(e_0)$ in $S$. 
%We will denote the type of $M$ by $\text{type}(M)$. 
\end{Definition}

\begin{Proposition}\label{P:NavelofJirreducible}
Let $M$ be a J-irreducible monoid of type $J$, and let $\Lambda$ be the cross-section lattice of $M$. 
Then $\Lambda$ has a navel if and only if $J=\emptyset$. 
In this case, the navel is equal to the minimal nonzero idempotent of $\Lambda$.
\end{Proposition}
\begin{proof}
Let us assume that $e$ is a navel in $\Lambda$.
By~\cite[Theorem 4.16]{PutchaRenner88}, the map \hbox{$\lambda^*: \Lambda\setminus \{0\} \to 2^S$} 
is injective, and furthermore, it is order preserving. 
Since by our assumption $\lambda^*(e) = \emptyset$, $e$ must be the minimal nonzero idempotent in $\Lambda$.
But then for any $s\in J$ we have $se =e s = e$. In other words, $\lambda_*(e) = J$. 
This argument is reversible. Thus, we see that $e$ is a navel in $\Lambda$ if and only if $J=\emptyset$.   
\end{proof}

By the dual argument, we have the following proposition whose proof is omitted. 
\begin{Proposition}\label{P:NavelofJcoirreducible}
Let $M$ be a J-coirreducible monoid of type $J$, and let $\Lambda$ be the cross-section lattice of $M$. 
Then $\Lambda$ has a navel if and only if $J=\emptyset$. In this case, the navel is equal to the maximal idempotent of $\Lambda\setminus \{1\}$.
\end{Proposition}

\begin{Remark}
For tautological reasons, in diagonal monoids, every idempotent can be viewed as a navel.
\end{Remark}

\begin{Proposition}\label{P:NavelofEnv}
Let $e=e_{I,J}$ be an idempotent in the cross-section lattice of an enveloping monoid. 
Then $e$ is a navel of $\Lambda$ if and only if $e=e_{S,\emptyset}$. 
In this case, we have isomorphisms 
\[
G_e^0 \cong G(eMe) \cong T_0\qquad \text{and}\qquad C_G(e) \cong T_0\times T_0.
\]
In particular, both of the local monoids $M_e$ and $eMe$ are affine torus embeddings. 
\end{Proposition}

\begin{proof}
Let $e_{I,J}$ be an idempotent in the cross-section lattice of an enveloping monoid. 
Then $\lambda(e_{I,J}) = \lambda_*(e_{I,J}) \sqcup \lambda^*(e_{I,J})$, where 
$\lambda_*(e_{I,J}) = J$ and $\lambda^*(e_{I,J}) = \{s\in S \setminus I: ss'=s's 
\text{ for every } s'\in J \}$. 
Thus, $\lambda(e_{I,J})= \emptyset$ if and only if 
1) $J=\emptyset$ and 2) $\{s\in S \setminus I: ss'=s's \text{ for every } s'\in J \} = \emptyset$.
Clearly, if $J=\emptyset$ and $I=S$, then 1) and 2) hold.
Conversely, if 1) and 2) hold, then $J=\emptyset$, and 
$\emptyset = \{s\in S \setminus I: ss'=s's \text{ for every } s'\in J \} = S\setminus I$,
and hence, $I=S$. This proves the first assertion. 

Next, by Lemma~\ref{L:Navel}, we know that the reductive groups 
$C_G(e)$, $G_e^0$, and $G(eMe)$ are tori; we will compute their dimensions.
Since $e= e_{S,\emptyset}$, both of the sublattices $\Lambda ( eMe )$ and $\Lambda (M_e)$ are isomorphic to the 
Boolean lattice on $S$, which is of height $|S| = \dim T_0$. Once again, the rest of the proof follows from Lemma~\ref{L:M_e}.
\end{proof}

Let $GeG$ be a $G\times G$-orbit in a reductive monoid $M$.
The following fibration is well-known,
\begin{align}\label{A:CR}
eC_G(e) \to GeG \to G/P \times G/P^-,
\end{align}
where $P$ (resp. $P^-$) is the right (resp. the left) stabilizer of $e$ in $G$; see, for example,~\cite[Lemmas 3.5 and 3.6]{CanRenner}.
In particular, $P\cap P^- = C_G(e)$. 
Now let $M$ be the enveloping semigroup of $G_0$, and let $e$ denote the navel of $\Lambda$. 
For simplicity, let us assume that $G_0$ is simply connected and of adjoint type. 
Since the rank of the idempotent $e$ is $|S|$, we know that $eC_G(e) = eT \cong T_0$.

We claim that the dimension of the $G\times G$-orbit $GeG$ in $M$ is equal to $\dim G - \dim T_0$. 
To see this, we refer to a theorem of Vinberg. 
Let $M^{spr}$ denote the open subset 
\[
M^{spr} := \bigsqcup_{f \geq e} G f G \subset M.
\]
Let $Z$ denote the center of $G$. Then $Z\cong T_0$. 
In~\cite[Theorem 7]{Vinberg1}, Vinberg shows that the geometric quotient $\pi: M^{spr} \to M^{spr}/Z$ exists, and furthermore, 
$M^{spr}/Z$ is isomorphic to the wonderful compactification of $G/Z \cong G_0$. 
The closed $G_0\times G_0$-orbit in $M^{spr}/Z$ is isomorphic to $G_0/B_0\times G_0/B_0^-$, where $B_0$
is a Borel subgroup of $G_0$ containing $T_0$. 
Since $\pi$ is $G\times G$-equivariant, the closed orbit of $M^{spr}$ is a fibration over $G_0/B_0\times G_0/B_0^-$ with fiber $Z$. 
But $G_0/B_0 \cong G_0/B_0^- \cong G/B$, therefore, we get the torus fibration 
\begin{align}\label{A:V}
Z \to GeG \to G/B\times G/B^-.
\end{align}
The dimension of $G/B$ is equal to the dimension of the unipotent radical of $B$, $\dim G/B = \dim U$.
As $\dim G = 2\dim U + \dim T$, we see from (\ref{A:V}) that $\dim GeG = \dim G - \dim T_0$. 
This finishes the proof of our claim. 
Now as a simple corollary of this fact, we see that for $M=\textrm{Env}(G_0)$ and for $e$ the navel of $M$, 
the fibrations (\ref{A:V}) and (\ref{A:CR}) are equal. 
In particular, the stabilizer of $e$ in $G\times G$ is a horospherical subgroup, that is to say, 
it is contains a maximal unipotent subgroup of $G\times G$.

We finish this section by computing the local monoids associated with the navels of the J-coirreducible and J-irreducible monoids of type $\emptyset$. 
Let $e$ denote the navel of a J-coirreducible monoid $M$.
Since the cross-section lattice of $M_e$ is 1 dimensional and since the unit group of $M_e$ is a torus, 
we see that $M_e \cong \mathbb{A}^1$. Likewise, the unit group of $eMe$ is an $|T|-1$-dimensional torus, $T'\subset T$. 
Therefore, $G(eMe) = eT'$, and $eMe = \overline{eT'}$.
By arguing in a similar way, we find that if $M$ is a J-irreducible monoid of type $\emptyset$ and $e$ is the navel of $M$, then 
$eMe = \mathbb{A}^1$ and $M_e \cong \overline{eT'}$, where $T'$ is a codimension one subtorus in $T$.

\section{Atomic Lattices}\label{S:Atomic}

We start with reviewing Vinberg's description of the parametrizing sets for $G\times G$-orbits $M=\textrm{Env}(G_0)$.
We will show that the lattice of $G\times G$-orbits in $M$ is an atomic lattice. 
%Then we will count the number $G\times G$-orbits in $M$ in the special case of $G_0=\textrm{SL}_{n+1}$. 

We maintain our notation from the preliminaries. 
In addition, we have the following notation: if $H$ is a closed subgroup of $G$, 
then $\OO(H)$ denotes $\OO(H):= \textbf{Hom}_{k\text{-grp}}(H,k^\times)$, the group of characters of $H$.
Then we set $\HO(H):=\textbf{Hom}_\Z ( \OO(H), \Q)$.
Now, recall that $T$ denotes the maximal torus in $G= (G_0\times T_0)/Z_0$, which is the unit group of $M$. 
Then $T \cong (T_0\times T_0)/Z_0$.
We will denote by $\alpha_1,\dots, \alpha_l$ the simple roots determined by $(G_0,B_0,T_0)$.
Let $\alpha_1^{\vee},\dots, \alpha_l^{\vee}$ denote 
the corresponding dual roots, and let $\CC$ denote the Weyl chamber,
\[
\CC= \{ f\in \HO(T) :\ f(\alpha_i^\vee) \geq 0 \ \text{ for } i =1,\dots, l\}.
\]
Let $\KK$ denote a closed, convex, polyhedral cone in $\HO(T)$ such that 
\begin{enumerate}
\item $-\alpha_i \in \KK$ for $i\in \{1,\dots, l\}$,
\item the cone $\KK \cap \CC$ generates $\HO(T)$.
\end{enumerate}
Since the Lie algebras of $G,G_0,T,T_0,Z_G$ are related to each other as follows
\[
\textrm{Lie}(G) = \textrm{Lie}(G_0) \oplus \textrm{Lie}(Z_G)
\ \text{ and }\  \textrm{Lie}(T) = \textrm{Lie}(T_0) \oplus \textrm{Lie}(Z_G),
\]
we see that 
\[
\HO(T) = \HO(T_0) \oplus \HO(Z_G) \ \text{ and } \
\CC = \HO(Z_G) \oplus \CC_0,
\]
where $\CC_0$ is the Weyl chamber of $(G_0,B_0,T_0)$.
We put $\MM:= \KK \cap \HO(Z_G)$. 
Clearly, $\MM \cap \CC_0 = \{ 0 \}$.
Note that $\MM$ is a pointed cone.
Moreover, if $\theta : Z_G \to T_0$ is a homomorphism such that $\theta \vert_{Z_0} = id$, then we have
\begin{enumerate}
\item $\KK:= \{ (\chi,\lambda) \in \HO(Z_G)\oplus \HO(T_0) :\ \chi-\theta^*(\lambda) \in \MM\}$;
\item the cone $\MM$ generates $\HO(Z_G)$;
\item $(\theta^*)^{-1} (\MM) \cap (-\CC_0) = \{0\}$.
\end{enumerate}
Finally, we see from these facts/definitions that $\KK \cap \CC \cong \MM \times \CC_0$.
For each subset $I \subseteq \{\alpha_1,\dots, \alpha_l\}$ we have a unique face of $\MM$,
denoted by $\MM_I$, which is spanned by $I$ as a convex cone. In a similar way, 
for each subset $J\subset \{\omega_1,\dots, \omega_l\}$, we have a unique face of $\CC_0$,
denoted by $(\CC_0)_J$, which is spanned by $J$. 
Therefore, the faces of $\KK \cap \CC$ are given by 
\begin{align}\label{A:Faces}
F_{I,J} := \{ (\chi, \lambda) 
 \in \HO(Z_G)\oplus \HO(T_0) :\ \chi-\theta^*(\lambda) \in  \MM_I,\ \lambda \in (\CC_0)_J \}.
\end{align}

\begin{Notation}
From now on, for positive integers, $l\in \Z_{\geq 1}$, we will use the shorthand $[l]:= \{1,\dots, l\}$. 
Also, by abusing of notation, if $I$ is a subset of the simple roots, or if $J$ is a 
subset of the set of fundamental weights, then we identify them by the sets of indices of the elements
that they contain, so, $I,J \subseteq [l]$. 
\end{Notation}
As before, we let $\varSigma$ denote the Dynkin diagram of $(G,B,T)$. For $I\subset [l]$,
we denote by $\varSigma_I$ the subdiagram of $\varSigma$ constituted by the vertices $v_i$,
where $i\in I$.

\begin{Definition}\label{D:essential}
A pair $(I,J)$ corresponding to a face $F_{I,J}$ of $\KK\cap \CC$ is called an {\em essential} if 
no connected component of the complement of $J$ is entirely contained in $I$. 
\end{Definition}

According to Vinberg~\cite[Theorem 6]{Vinberg1}, the $G\times G$-orbits in $M$ are in one-to-one correspondence with 
the essential faces of $\KK\cap \CC$; the inclusion order between the closures of the $G\times G$-orbits 
is equivalent to the inclusion order on essential faces:  
\begin{align}\label{A:Flattice}
F_{I_1,J_1} \subseteq F_{I_2,J_2} \iff I_1 \subseteq I_2 \ \text{ and } J_1 \subseteq J_2,
\end{align}
where $I_i,J_i \subseteq [l]$ for $i\in \{1,2\}$.
We will denote lattice of all faces of $\KK\cap \CC$ by $\overline{\mathcal{L}}$;
the sublattice of the essential faces will be denoted by $\mathcal{L}$.

\begin{Remark}\label{R:isomorphic}
The association $F_{(I,J)} \leftrightsquigarrow e_{([l]\setminus I , [l]\setminus J)}$ is 
a lattice isomorphism between $\mathcal{L}$ and the cross-section lattice $\Lambda$. 
This isomorphism extends to give an isomorphism between the lattices $\overline{\mathcal{L}}$ and $E(\overline{T})$. 
\end{Remark}

\begin{Lemma}\label{L:wedgevee}
Let $F_{I_1,J_1}$ and $F_{I_2,J_2}$ be two essential faces from $\mathcal{L}$. 
Then we have 
\begin{align*}
F_{I_1,J_1} \wedge F_{I_2,J_2} = F_{I_1\cap I_2, J_1\cap J_2} \ \text{ and } \ 
F_{I_1,J_1} \vee F_{I_2,J_2} = F_{I_1\cup I_2, J_1\cup J_2 \cup N},
\end{align*}
where $N$ is the union of the connected components $A_1,\dots , A_r$ of $[l]\setminus J_1\cup J_2$ such that $A_j\subseteq I_1\cup I_2$
for $j\in \{1,\dots, r\}$. 
\end{Lemma}
\begin{proof}
The first equality is an immediate consequence of the ``definition'' in (\ref{A:Flattice}).
We proceed to prove the second equality. 
Clearly, if $F_{A,B} \geq F_{I_i,J_i}$ for $i\in \{1,2\}$, then $A\supseteq I_1\cup I_2$ and $B\supseteq J_1\cup J_2$. 
Therefore, in order for $F_{A,B}$ be equal to $F_{I_1,J_1} \vee F_{I_2,J_2}$, first, we must have $A= I_1\cup I_2$. 
Secondly, the condition that no connected component of the complement of $B$ is entirely contained in $I_1\cup I_2$ must be satisfied. 
This condition is minimally satisfied, if we adjoint to $J_1\cup J_2$ the components $A_1,\dots, A_r$ of $[l]\setminus J_1\cup J_2$ such that $A_i \subseteq I_1\cup I_2$. But this is exactly our second claim, hence, the proof is finished. 
\end{proof}

Let $L$ be a lattice with a minimal element $\hat{0}$. 
An element $x$ in $L$ is called an {\em atom} if $x$ covers $\hat{0}$. 
It is easy to see that in $\mathcal{L}$ the atoms are given by $F_{\{i\},\emptyset}$ ($i\in [l]$)
and $F_{\emptyset, \{j\}}$ ($j\in [l]$).
A lattice $L$ is said to be {\em atomic} if every element $x\in L$ is a join of atoms. 
It is pointed out by Vinberg~\cite[Section 0.6]{Vinberg1} that $\KK\cap \CC$ is a {\em simplicial cone},
that is, a cone generated by linearly independent vectors. 
It is well-known that the lattice of faces of a simplicial cone is a Boolean lattice.
Therefore, $\overline{\mathcal{L}} \cong E(\overline{T})$ is a Boolean lattice. 
In particular, $E(\overline{T})$ is an atomic lattice. 
Note that since $\dim T = 2|S|$ is the height of $E(\overline{T})$, it has exactly $2|S|$ atoms. 
Therefore, the sets of atoms of $\Lambda$ and $E(\overline{T})$ are equal. 
This argument shows that $\Lambda$ is an atomic lattice. 
In the next proposition, we prove this result more directly. 

%We are now ready to show that $\mathcal{L}$ is an atomic lattice.
\begin{Proposition}\label{P:Atomic}
Let $F_{I,J}$ be an element from $\mathcal{L}$.
Then it has the following decomposition:
\[
F_{I,J}= \bigvee_{i\in I} F_{\{i\},\emptyset} \vee \bigvee_{j\in J} F_{\emptyset, \{j\}}.
\]
In particular, $\mathcal{L}$, hence, the cross-section lattice $\Lambda$ of $M$, is an atomic lattice.
\end{Proposition}
\begin{proof}
Without loss of generality we may assume that $I\neq [l]$. Indeed, 
there is a unique face $F_{I,J}$ with $I=[l]$; it is the top element of $\mathcal{L}$. 
Now, both of the faces $F_{I,\emptyset }$ and $F_{\emptyset,J} $ are elements in $\mathcal{L}$. 
It follows from Lemma~\ref{L:wedgevee} that $F_{I,J} = F_{I,\emptyset } \vee F_{\emptyset,J}$. 
It also follows from  Lemma~\ref{L:wedgevee} that $F_{I,\emptyset }=\bigvee_{i\in I} F_{\{i\},\emptyset}$
and $F_{\emptyset,J}=\bigvee_{j\in J} F_{\emptyset, \{j\}}$. This finishes the first part of the proof. 
Our second claim follows from Remark~\ref{R:isomorphic}.
\end{proof}

\section{A Generating Function}\label{S:Generatingfunction}

In this section, we will determine the generating series of the number $d_n$ of $G\times G$-orbits in 
$\textrm{Env}(G_0)$ for $G_0 := \textrm{SL}_{n+1}$. 
The Dynkin diagram of $G_0$, which we denote by $\varSigma_n$, 
has $n$-nodes labeled with the simple roots $\alpha_1,\dots, \alpha_n$. 
By Definition~\ref{D:essential}, our problem is equivalent to counting 
pairs $(I,J)$ such that 
\begin{align}\label{A:equivalent}
\text{every connected component of $J$ intersects $I$.}
\end{align}

First, we find a recurrence for the $d_n$. Clearly, $d_1 = 3$, so, we assume that $n>1$.
We split our count into two disjoint sets:

\begin{enumerate}
\item[(1)] $D_n'$: the set of pairs $(I,J)$ satisfying (\ref{A:equivalent}) and $J=\emptyset$;
\item[(2)] $E_n$: the set of pairs $(I,J)$ satisfying (\ref{A:equivalent}) and $J\neq \emptyset$.
\end{enumerate}
Clearly, $d_n = |D_n'|+ |E_n |$.
If $J\neq \emptyset$, then $I\neq \emptyset$. But if $J=\emptyset$, then $I$ 
can be any of the $2^n$ subsets of $\varSigma_n$, hence, $|D_n'|= 2^n$.

We proceed to find a formula (recurrence) for $e_n$. 
Once again, we split our problem into two parts: 
\begin{enumerate}
\item[(2.1)] counting $(I,J) \in E_n$ such that $1 \notin J$;

\item[(2.2)] counting $(I,J)\in E_n$ such that $1 \in J$. 
\end{enumerate}

In the former case, $1$ may, or may not, belong to $I$. 
In both of these cases, by removing $\alpha_1$ from $\varSigma_n$,
and relabeling the nodes, we obtain a pair $(I',J')$ (where $J'=J$) in $E_{n-1}$. 
Conversely, by appending $\alpha_1$ to $\varSigma_{n-1}$ as the new first node,
from any pair $(I',J')$ in $E_{n-1}$, we obtain two new pairs $(I,J)$ 
in $E_n$ such that $J=J'$.  Therefore, the number of such pairs is given by $2|E_{n-1}|$.

Now, in the latter case, we look at the following two disjoint situations: 
\begin{enumerate}
\item[(2.2.1)] counting $(I,J) \in E_n$ such that $[n] = J$;
\item[(2.2.2)] counting $(I,J)\in E_n$ such that $[s] \subseteq J$, where $s\in \{1,\dots, n-1\}$,
and $s+1\notin J$. 
\end{enumerate}

In the former case, $I$ can be any nonempty subset of $[n]$, therefore, we get a contribution 
of $2^n-1$ from (2.2.1). 
We proceed with the latter case. To this end, let us fix a subset $J$ of $[n]$ such that 
$[s] \subseteq J$, where $s\in \{1,\dots, n-1\}$, and $s+1\notin J$. 
For such $J$, we will analyze the possibilities for $(I,J)$.
The intersection $I\cap [s]$ is allowed to be any nonempty subset of $[s]$. 
Also, $s+1$ may or may not be an element of $I$. Next, we look at 
the tails, namely, the intersections $I\cap \{s+2,\dots, n\}$ and $J\cap \{s+2,\dots, n\}$.
Clearly, the intersection $J\cap \{s+2,\dots, n\}$ might be empty, or not. If it is empty,
then $I\cap \{s+2,\dots, n\}$ can be chosen arbitrarily, so, it gives $2^{n-(s+1)}$ possibilities. 
If $J\cap \{s+2,\dots, n\}$ is nonempty, then $I\cap \{s+2,\dots, n\}$ can be chosen in
one of the $|E_{n-(s+1)}|$ possible ways. 
Thus, the possibilities for $(I,J)$ are exhausted, and we arrive at a formula for the 
cardinality of $E_n$, 
\begin{align}\label{A:recurrence for e}
e_n = 2 e_{n-1} + 2^n-1 + \sum_{s=1}^{n-1} 2 (2^s-1) (e_{n-(s+1)} + 2^{n-(s+1)}).
\end{align}
We set $e_0:=0$ and $e_1:=1$.
By reorganizing the right hand side of (\ref{A:recurrence for e}), we obtain the following lemma. 
\begin{Lemma}\label{L:e's}
For every positive integer $n$ with $n>1$, the following recurrence formula hold:
\begin{align}\label{A:recurrence for e 2}
e_n = 2 e_{n-1} + \sum_{s=1}^{n-1}  (2^{s+1}-2) e_{n-(s+1)} + (n-1) 2^n + 1.
\end{align}
\end{Lemma}

It is easy to verify (by hand) and by Lemma~\ref{L:e's} that $e_2= 7$ and that $e_3=33$. 
Next, we will determine a closed formula for the generating series, 
\begin{align}\label{A:Ex}
E(x):=\sum_{n\geq 0} e_n x^n.
\end{align}
Let us first introduce the notation $a_i := 2^{i+2}-2$ for $i\in \N$. 
By modifying the limits of the summation on the right hand side of eqn. (\ref{A:recurrence for e 2}), 
we get 
\begin{align}\label{A:recurrence for e 3}
e_n = 2 e_{n-1} + \sum_{i=0}^{n-2}  a_i e_{(n-2)-i} + (n-1) 2^n + 1.
\end{align}
Thus, by multiplying both sides of eqn. (\ref{A:recurrence for e 3}) by $x^n$ and then 
by taking the sum over $n$ with $n\geq 2$,
we obtain
\begin{align*}
E(x) -  x &= 2 x E(x)  + 
x^2 \sum_{n\geq 2} \left( \sum_{i=0}^{n-2}  a_i e_{(n-2)-i} \right)x^{n-2}
+ \sum_{n\geq 2} ((n-1)2^n +1)x^n \\
&= 2 x E(x)  + x^2 E(x) \sum_{n\geq 0}  a_n x^n  + \sum_{n\geq 2} ((n-1)2^n +1)x^n \\
&= 2 x E(x)  + x^2 E(x)\left( \frac{4}{1-2x}- \frac{2}{1-x}\right)  + \sum_{n\geq 2} ((n-1)2^n +1)x^n. \\
\end{align*}
By solving for $E(x)$, we obtain
\begin{align*}
E(x) &= \frac{x+ \sum_{n\geq 2} ((n-1)2^n +1)x^n }{1-2x- x^2 \left( \frac{4}{1-2x}- \frac{2}{1-x}\right) }.
\end{align*}
The denominator is given by 
\begin{align*}
1-2x- x^2 \left( \frac{4}{1-2x}- \frac{2}{1-x}\right)  = \frac{ 1-5x + 6x^2 -4x^3}{(1-2x)(1-x)}.
\end{align*}
But the numerator is easy to compute as well,
\begin{align*}
x+ \sum_{n\geq 2} ((n-1)2^n +1)x^n  &= x+ \sum_{n\geq 2} n(2x)^n - \sum_{n\geq 2} (2x)^n + \sum_{n\geq 2} x^n \\
 &= x+ 2x \left( \frac{1}{(1-2x)^2}-1 \right) - \frac{1}{1-2x} +1 +2x + \frac{1}{1-x}-1-x \\
    &= \frac{ x   }{(1-2x)^2(1-x)}.
\end{align*}
Thus we have a cleaner formula for $E(x)$,
\begin{align}\label{A:E}
E(x) = \frac{ \frac{ x   }{(1-2x)^2(1-x)} }{  \frac{ 1-5x + 6x^2 -4x^3}{(1-2x)(1-x)}} 
= \frac{ x   }{(1-2x)(1-5x + 6x^2 -4x^3)}. 
\end{align}

We are now ready to prove our formulation of the generating series for $d_n$.
For convenience, we set $d_0 :=1$. 

\begin{Theorem}\label{T:numbers}
Let $G$ denote $(\textrm{SL}_n\times T_0)/Z_{ \textrm{SL}_n}$, where $T_0$ is the maximal torus in $ \textrm{SL}_n$,
and $Z_{ \textrm{SL}_n}$ is center of $ \textrm{SL}_n$. 
The generating series of the number $d_n$, $n=0,1,2,\dots$ of $G\times G$-orbits 
in the enveloping monoid of $\textrm{SL}_n$ is given by 
\begin{align}
\sum_{n\geq 0} d_n x^n  
%=\frac{x  }{(1-2x)(1-5x + 6x^2 -4x^3)}  + \frac{1}{1-2x}
= \frac{ 1-2x+2x^2   }{(1-5x + 6x^2 -4x^3)}   = 1+ 3x + 11x^2 + 41x^3 + 151x^4 + 553 x^5 + O(x^6).
\end{align}
\end{Theorem}
\begin{proof}
We already mentioned that $d_n = 2^n + e_n \ \text{ for all } n \in \Z_{\geq 2}$.
Since $\sum_{n\geq 0} 2^nx^n = 1/(1-2x)$, the proof follows from formula (\ref{A:E}) after a simple calculation.
\end{proof}

\bibliographystyle{plain}
\bibliography{referenc}

\end{document}